\documentclass[10pt]{amsart}
\usepackage{amsmath, latexsym, amsthm, amsfonts,bm,amssymb} 
\usepackage{appendix}
\usepackage{natbib} 
\usepackage{url} 
\usepackage{enumerate}
\usepackage{mathtools} 

\bibpunct{(}{)}{;}{a}{}{,} 

\theoremstyle{plain}
\newtheorem{thm}{Theorem}

\newtheorem{lemma}{Lemma}

\theoremstyle{remark}
\newtheorem{rem}{Remark}
\newtheorem{exam}{Example}

\def\ex{{\rm {\mathbb E\,}}}

\newcommand{\pp}{\mathbb{P}}

\newcommand{\rr}{\mathbb{R}}

\allowdisplaybreaks

\begin{document}

\title[Non-parametric estimation for Poisson point processes]{A note on non-parametric Bayesian estimation for Poisson point processes}

\author{Shota Gugushvili}
\address{Mathematical Institute\\
Leiden University\\
P.O.\ Box 9512\\
2300 RA Leiden\\
The Netherlands}
\email{shota.gugushvili@math.leidenuniv.nl}

\author{Peter Spreij}
\address{Korteweg-de Vries Institute for Mathematics\\
Universiteit van Amsterdam\\
P.O.\ Box 94248\\
1090 GE Amsterdam\\
The Netherlands}
\email{spreij@uva.nl}

\thanks{The research leading to these results has received funding from the European Research Council under ERC Grant Agreement 320637. Support by The Netherlands Organisation for Scientific Research (NWO) is likewise acknowledged.}



\begin{abstract}

We derive the posterior contraction rate for non-parametric Bayesian estimation of the intensity function of a Poisson point process.

\end{abstract}

\date{\today}

\maketitle

\section{Introduction}
\label{intro}
Poisson point processes (see e.g. \cite{kingman93}) are among the basic modelling tools in areas as different as astronomy, biology, image analysis, reliability theory, medicine, physics, and others. A Poisson point process $X$ on the compact metric space $\mathcal{X}$ with the Borel $\sigma$-field $\mathcal{B}(\mathcal{X})$ of its subsets is a random integer-valued measure on $\mathcal{X}$ (we assume the underlying (complete) probability space $(\Omega,\mathcal{F},\mathbb{Q})$ in the background), such that
\begin{enumerate}
\item for any disjoint subsets $B_1,B_2,\ldots,B_m\in\mathcal{B}(\mathcal{X}),$ the random variables $X(B_1),X(B_2),\ldots,X(B_m)$ are independent, and
\item for any $B\in\mathcal{B}(\mathcal{X}),$ the random variable $X(B)$ is Poisson distributed with parameter $\Lambda(B),$ where $\Lambda$ is a finite measure on $(\mathcal{X},\mathcal{B}(\mathcal{X})),$ called the intensity measure of the process $X.$
\end{enumerate}
Intuitively, the process $X$ can be thought of as random scattering of points in $\mathcal{X},$ where scattering occurs in a special way determined by properties (i)--(ii) above.

In practical applications knowledge of the intensity $\Lambda$ is of importance. The latter typically cannot be assumed known beforehand and has to be estimated based on the observational data on the process $X.$ A popular assumption in the literature (see e.g.\ pp.~96--97 in \cite{karr86}, or the references on p.~263 in \cite{kutoyants98}) is that one has independent observations $X_1,\ldots,X_n$ on the process $X$ over $\mathcal{X}$ at his disposal, on basis of which an estimator of $\Lambda$ has to be constructed. We will denote for brevity $X^{(n)}=(X_1,X_2,\ldots,X_n).$ In case $\Lambda$ is absolutely continuous with respect to some dominating measure $\mu$ and has a density $\lambda,$ one might also be interested in estimation of $\lambda.$ We will assume that $\mu$ is a finite measure on $\mathcal{X},$ so that without loss of generality it can be taken to be a probability measure, and we will call $\lambda$ the intensity function.

From now on we concentrate on estimation of the intensity function. Two broad approaches to estimation of $\lambda,$ parametric and non-parametric, can be discerned in the literature. In the parametric approach, one assumes that the unknown intensity function $\lambda$ can be parametrised by a finite-dimensional parameter $\theta$ (where, for instance, $\theta$ ranges in some subset $\Theta$ of $\mathbb{R}^p$), so that $\lambda=\lambda_{\theta},$ and the corresponding statistical experiment generated by $X^{(n)}$ is denoted by $(\mathcal{X}^{n},\mathcal{B}(\mathcal{X}^n),\{\mathbb{P}_{\theta}^{(n)},\theta\in\Theta\}).$ The goal is to estimate the `true' parameter $\theta_0$ on the basis of the sample $X^{(n)}.$ In the non-parametric approach to estimation of $\lambda$ no such assumptions are made. Instead, one assumes, for instance, that $\lambda$ belongs to some class $\Theta$ of functions possessing given smoothness properties (the statistical experiment generated by $X^{(n)}$ is $(\mathcal{X}^{n},\mathcal{B}(\mathcal{X}^n),\{\mathbb{P}_{\lambda}^{(n)},\lambda\in\Theta\})$), and the goal is to estimate the `true' intensity function $\lambda_0.$ 

In this note we are interested in non-parametric estimation of the intensity function $\lambda_0.$ In the particular case $\mathcal{X}=[0,1]^d$ a kernel-type estimator of $\lambda_0$ has been studied in detail in Section 6.2 in \cite{kutoyants98}; see also p.~263 there for further references, as well as \cite{diggle85}, where a practical implementation is discussed in the univariate setup in a closely related problem of estimation of the intensity of a stationary Cox process based on a single realisation of the process over the interval $[0,T]$ (see e.g.\ Chapter 6 in \cite{kingman93} or Chapter 1 in \cite{karr86} for Cox processes). In particular, it is shown in \cite{kutoyants98} that the kernel estimator is optimal in the minimax sense over the class of $\beta$-H\"older-regular intensity functions.

Here we will take an alternative, non-parametric Bayesian approach to estimation of $\lambda_0.$ We note that a non-parametric Bayesian approach to estimation of the intensity measure $\Lambda$ of a Poisson point process was studied in \cite{lo82}. Papers dealing with non-parametric Bayesian intensity function estimation include, among others, \cite{adams09}, \cite{arjas98} and \cite{moller98}, and concentrate primarily on computational aspects; an extensive bibliography is given in \cite{gugu18}. Advantages of a non-parametric Bayesian approach over the classical kernel method are succintly summarised in the discussion given in \cite{adams09}. One obvious drawback of the kernel estimator from \cite{kutoyants98} is that it is inconsistent on the boundary of the set $\mathcal{X}=[0,1]^d$ on which the process is defined, and in practice the estimator will also behave poorly in the regions close to the boundary (this has to do with the well-known boundary bias problem of the kernel estimator). Simulation examples in \cite{adams09} demonstrate that even after the correction for edge effects following the method in \cite{diggle85} and with an optimal choice of the bandwidth parameter, the kernel method is still  outperformed by various Bayesian approaches. Secondly, if a kernel of order higher than two is used, the kernel method will not even yield a necessarily non-negative estimate of the intensity function, while if one restricts attention to e.g.\ a normal kernel, the resulting estimator will have a sub-optimal convergence rate for functions that are smoother than twice differentiable functions. The non-parametric Bayesian approach we advocate does not suffer from these drawbacks. Another often highlighted problem with kernel estimators is practical bandwidth selection, which is especially, at times hopelessly complicated for multivariate functions.

In the Bayesian approach to estimation of $\lambda_0$ one puts a prior $\Pi$ on $\lambda_0,$ which might be thought of as reflecting one's prior knowledge or belief in $\lambda_0.$ In more formal terms this is a measure $\Pi$ defined on the parameter set $\Theta$ equipped with a $\sigma$-field $\sigma(\Theta),$ and one assumes that $\lambda_0\in\Theta.$ The set $\Theta$ equipped with a certain $\sigma$-field $\sigma(\Theta)$ is a set of finite-valued functions defined on $[0,1]^d,$  which we for technical reasons assume to be uniformly bounded away from zero. Then by Proposition 6.11 in \cite{karr86} or Theorem 1.3 in \cite{kutoyants98}, for any $\lambda\in\Theta,$ the law $\mathbb{P}_{\lambda}$ of $X$ under the parameter value $\lambda$ admits a density $p_{\lambda}$ with respect to the measure $\mathbb{P}_{\mathrm{sp}}$ induced by a standard Poisson point process with intensity measure $\mathrm{d}\Lambda_{\mathrm{sp}}(x)=\mathrm{d}\mu(x).$ This density is given by
\begin{equation*}
p_{\lambda}(\xi)=\exp\left( \int_{\mathcal{X}} \log \lambda(x)\mathrm{d}\xi(x) - \int_{ \mathcal{X}  } [\lambda(x)-1] \mathrm{d}\mu(x) \right),
\end{equation*}
where $\xi=\sum_{i=1}^m \delta_{x_i}$ is a realisation of $X$ (here $\delta_{x_i}$ denotes the Dirac measure at $x_i$) and
\begin{equation*}
\int_{\mathcal{X}} \log \lambda(x)\mathrm{d}\xi(x) = \sum_{i=1}^m \log (\lambda(x_i)).
\end{equation*}
Using independence of the $X_i$'s, it follows that the likelihood $L_{\lambda}(X^{(n)})$ in $\lambda$ can be written as
\begin{equation}
\label{likelih}
L_{\lambda}(X^{(n)})=\prod_{i=1}^n \exp\left( \int_{ \mathcal{X} } \log \lambda(x)\mathrm{d}X_{i}(x) - \int_{ \mathcal{X}  } [\lambda(x)-1] \mathrm{d}\mu(x) \right).
\end{equation}
Assuming joint measurability of $p_{\lambda}(\xi)$ in $(\xi,\lambda),$ so that the integrals below make sense, Bayes' formula gives the posterior measure of any measurable set $A\in\sigma(\Theta)$ through
\begin{equation*}
\Pi(A|X^{(n)})=\frac{ \int_A L_{\lambda} (X^{(n)}) \mathrm{d}\Pi(\lambda) }{ \int_{ \Theta } L_{\lambda} (X^{(n)}) \mathrm{d}\Pi(\lambda) }.
\end{equation*}
Transition from the prior to the posterior can be thought of as updating our prior opinion on $\lambda_0$ upon seeing the data $X^{(n)}.$

In this paper we will study the rate of convergence of the posterior distribution $\Pi(\cdot|X^{(n)})$ under $\mathbb{P}_{\lambda_0}^{(n)},$ where $\mathbb{P}_{\lambda_0}^{(n)}$ denotes the law of $X^{(n)}$ under the true parameter value $\lambda_0.$ The idea, informally speaking, is that with the sample size $n$ growing indefinitely, the Bayesian approach should be able to recognise the true $\lambda_0$ with increasing accuracy. This can be formalised by requiring, for instance, that for any fixed neighbourhood $A$ of $\lambda_0,$ $\Pi(A^c|X^{(n)})\rightarrow 0$ in $\mathbb{P}_{\lambda_0}^{(n)}$-probability, or, in words, by requiring that with the Bayesian approach with a prior $\Pi,$ most of the posterior mass must eventually concentrate around the true parameter value $\lambda_0.$ More generally, one might take a sequence of shrinking neighbourhoods $A_n$ of $\lambda_0$ and ask what is the fastest rate, at which the neighbourhoods $A_n$ can shrink, while still capturing most of the posterior mass (the precise definition will be given below). The case for such an approach to the study of non-parametric Bayesian techniques is made e.g.\ in \cite{diaconis86}, while several recent references dealing with establishing posterior convergence rates under broad conditions in various statistical settings are \cite{ghosal00}, \cite{ghosal01}, \cite{ghosal07} and \cite{vdv08a}. The rate, at which the neighbourhoods $A_n$ shrink, can be thought of as an analogue of the convergence rate of a frequentist estimator. The analogy can be made precise in the sense that contraction of the posterior distribution at a certain rate implies existence of a Bayes point estimate with the same convergence rate (in the frequentist sense); see e.g.\ Theorem 2.5 in \cite{ghosal00} and the discussion on pp. 506--507 there. Given the widespread use of non-parametric Bayesian procedures for intensity function estimation (see e.g.\ the references cited above, as well as an overview paper \cite{moller07}), theoretical justification of the non-parametric Bayesian approach to inference in Poisson point processes appears timely to us.

The rest of the paper is organised as follows: in the next section we state the problem we are interested in in greater detail and provide a general result on the posterior contraction rate in our problem with the prior based on a transformation of a Gaussian process. In Section \ref{priors} we consider concrete examples of the prior and compute the posterior contraction rates explicitly. Finally, Appendix \ref{appendix} contains the proof of the technical lemma used in the proof of our main theorem.

\section{Main result}
\label{results}

In order to study the contraction rate of the posterior distribution in our setting, we first need to specify the suitable neighbourhoods $A_n$ of $\lambda_0,$ for which this will be done. The Hellinger distance $h(\mathbb{P}_{\lambda_1},\mathbb{P}_{\lambda_2})$ between two probability laws $\mathbb{P}_{\lambda_1}$ and $\mathbb{P}_{\lambda_2}$ is defined as
\begin{align*}
h(\mathbb{P}_{\lambda_1},\mathbb{P}_{\lambda_2})&=\left\{ \int ( \mathrm{d}\mathbb{P}_{\lambda_1}^{1/2} - \mathrm{d}\mathbb{P}_{\lambda_2}^{1/2} )^2 \right\}^{1/2}\\
&=\left\{ \int ( p_{\lambda_1}^{1/2} -  p_{\lambda_2}^{1/2} )^2 \mathrm{d}\mathbb{P}_{\mathrm{sp}} \right\}^{1/2}.
\end{align*}
Here, as in Section \ref{intro}, we assume that $\lambda_i$'s are bounded away from zero and infinity, which yields in particular the second equality in the above display. The Hellinger distance is one of the popular discrepancy measures between two probability laws. The Hellinger distance can also be used to define the pseudo-distance $\rho(\lambda_1,\lambda_2)$ between parameters $\lambda_1$ and $\lambda_2$ by setting
\begin{equation*}
\rho(\lambda_1,\lambda_2)=h(\mathbb{P}_{\lambda_1},\mathbb{P}_{\lambda_2}).
\end{equation*}
Thus $\lambda_1$ and $\lambda_2$ are close to each other if the corresponding laws $\mathbb{P}_{\lambda_1}$  and $\mathbb{P}_{\lambda_1}$ are in Hellinger distance. We also introduce two further discrepancy measures: the Kullback-Leibler divergence $\mathrm{KL}(\mathbb{P}_{\lambda_1},\mathbb{P}_{\lambda_2})$ between two probability laws $\mathbb{P}_{\lambda_1}$ and $\mathbb{P}_{\lambda_2}$ is defined as
\begin{align*}
\mathrm{KL}(\mathbb{P}_{\lambda_1},\mathbb{P}_{\lambda_2})&= \int \log \left(\frac{ \mathrm{d}\mathbb{P}_{\lambda_1}}{ \mathrm{d}\mathbb{P}_{\lambda_2} } \right) \mathrm{d}\mathbb{P}_{\lambda_1} \\
&=\int p_{\lambda_1} \log \left(\frac{ p_{\lambda_1}}{ p_{\lambda_2} } \right)  \mathrm{d}\mathbb{P}_{\mathrm{sp}} ,
\end{align*}
while the discrepancy $\mathrm{V}$ is defined through
\begin{align*}
\mathrm{V}(\mathbb{P}_{\lambda_1},\mathbb{P}_{\lambda_2})&= \int \left( \log \left(\frac{ \mathrm{d}\mathbb{P}_{\lambda_1}}{ \mathrm{d}\mathbb{P}_{\lambda_2} } \right) - \mathrm{KL}(\mathbb{P}_{\lambda_1},\mathbb{P}_{\lambda_2}) \right)^2 \mathrm{d}\mathbb{P}_{\lambda_1} \\
&=\int p_{\lambda_1} \left( \log \left(\frac{ p_{\lambda_1}}{ p_{\lambda_2} } \right) - \mathrm{KL}(\mathbb{P}_{\lambda_1},\mathbb{P}_{\lambda_2}) \right)^2 \mathrm{d}\mathbb{P}_{\mathrm{sp}}.
\end{align*}
This can be thought of as the Kullback-Leibler `covariance'. Both quantities are well-defined, because under our standing assumption that $\lambda_1$ and $\lambda_2$ are bounded away from zero, the corresponding laws $\mathbb{P}_{\lambda_1}$ and $\mathbb{P}_{\lambda_2}$ are equivalent.

We will derive the posterior convergence rate by taking the neighbourhoods $A_n$ of $\lambda_0$ to be balls of appropriate radii in the pseudo-distance $\rho,$ see below. This is a reasonable choice, see e.g.\ \cite{ghosal00}. In fact, utilising Lemma 1.5 in \cite{kutoyants98}, under our conditions the sets $A_n$ can be seen to be contained in fixed multiples of the $L_2(\mu)$-balls with the same radii as $A_n.$

We need to specify the prior $\Pi.$ Priors based on stochastic processes are widely used in Bayesian statistics. In particular, priors based on Gaussian processes are a popular choice both in the statistics and machine learning communities, see e.g.\ \cite{rasmussen06}, as well as \cite{vdv08a} for additional references. For our purposes, a zero-mean Gaussian process $W=(W(x):x\in\mathcal{X})$ is a collection of random variables $W(x)$ indexed by $\mathcal{X}$ and defined on the common probability space $(\widetilde{\Omega},\widetilde{\mathcal{F}},\widetilde{\mathbb{P}})$, such that the finite-dimensional distributions of $W$ are zero-mean multivariate normal distributions. The latter are determined by the covariance function $K:\mathcal{X}\times\mathcal{X}\rightarrow\mathbb{R},$ defined by
\begin{equation*}
K(x,y)=\widetilde{\ex}[W(x) W(y)], \quad x,y\in\mathcal{X},
\end{equation*}
where $\widetilde{\ex}$ denotes the expectation with respect to the measure $\widetilde{\mathbb{P}}.$ For all the necessary definitions and properties of Gaussian processes with a view towards applications in non-parametric Bayesian statistics that are used in this work, see \cite{vdv08a} and \cite{vdv08b}. An introductory treatment of probabilistic properties of Gaussian processes and random elements is given in \cite{lifshits12}.

Assume that $W$ is a zero-mean Gaussian process with bounded sample paths $x\mapsto W(x)$ and let $\kappa>0$ be a fixed constant. Pick a measurable function $g:\rr\rightarrow[\kappa,\infty),$ and define the process $Z^{(W)}=\left(Z^{(W)}(x):{x\in\mathcal{X}}\right)$ through
\begin{equation}
\label{pz}
Z^{(W)}(x)=g( W(x) ), \quad x\in\mathcal{X}.
\end{equation}
Realisations of $W$ will be denoted by lowercase letters, such as $w$ and $v.$ The corresponding realisations of $Z^{(W)}$ will be denoted by $z^{(w)}$ and $z^{(v)}.$ Our prior $\Pi$ will be the law of the process $Z^{(W)},$ which implicitly defines our parameter set $\Theta.$ The only reason for using the constant $\kappa>0$ in the definition of the function $g$ is to make sample paths of $Z^{(W)}$ uniformly bounded away from zero, which allows one to avoid complications when manipulating in the proofs expressions involving the likelihood \eqref{likelih}. The constant $\kappa$ can be taken to be arbitrarily small, and in a practical implementation one can ignore it altogether and simply take $g>0.$

Sample paths of those processes $W$ that are important for applications typically possess some smoothness properties, e.g.\ are continuous, and $W$ can also be viewed as a Borel-measurable random element taking values in a Banach space $(\mathbb{B},\|\cdot\|_{\infty})$ for some $\mathbb{B}\subset\ell^{\infty}(\mathcal{X});$  cf.\ Example 2.4 in \cite{lifshits12}. By Lemma 5.1 in \cite{vdv08b}, the support of $W,$ i.e.\ the smallest closed set $\mathbb{B}_0\subset \mathbb{B},$ such that $\widetilde{\mathbb{P}}(W\in \mathbb{B}_0)=1,$ is the closure  in $\mathbb{B}$ of  the reproducing kernel Hilbert space (RKHS) $(\mathbb{H},\|\cdot\|_{\mathbb{H}})$ attached to $W.$
It can be shown that this RKHS can be identified with the completion of the set of maps
\begin{equation*}
x\mapsto \sum_{i=1}^k \alpha_i K(y_i,x)=\widetilde{\ex} [W_x H], \quad H=\sum_{i=1}^k \alpha_i W_{y_i},
\end{equation*}
under the inner product
\begin{equation*}
\langle \widetilde{\ex} [W_{\cdot} H_1], \ex [W_{\cdot} H_2] \rangle_{\mathbb{H}} = \widetilde{\ex} [H_1 H_2].
\end{equation*}
Here the $\alpha_i$'s range over $\mathbb{R}$ and $k$ ranges over $\mathbb{N}.$ The support  of the process $Z^{(W)}$ can then be described through this characterisation of the support of the process $W.$

Let $N(\varepsilon,B,f)$ denote the minimum number of balls of radius $\varepsilon$ needed to cover a subset $B$ of a metric space with metric $f.$ This is the $\varepsilon$-covering number of $B.$

Our main result is based on an application of Theorem 2.1 in \cite{ghosal01} (which is a slight modification of Theorem 2.1 in \cite{ghosal00}) and Theorem 2.1 from \cite{vdv08a}. These are provided  below for the reader's convenience in an adapted form (the statement of the theorem in \cite{ghosal01} uses a slightly different definition of the discrepancy $V,$ but the theorem is valid also with our definition, cf.\ Theorem 1 in \cite{ghosal07} and the arguments on p.~196 there).

\begin{thm}[\cite{ghosal01}]
\label{thm2.1ghosal01}
Suppose that for positive sequences $\overline{\varepsilon}_n,\widetilde{\varepsilon}_n\rightarrow 0,$ such that $n\min(\overline{\varepsilon}_n^2,\widetilde{\varepsilon}_n^2)\rightarrow\infty,$ constants $c_1,c_2,c_3,c_4>0$ and sets $\Theta_n\subset \Theta,$ we have
\begin{align}
\log N(\overline{\varepsilon}_n,\Theta_n,\rho) & \leq c_1 n\overline{\varepsilon}_n^2,\label{c1}\\
\Pi( \Theta\setminus \Theta_n ) & \leq c_3 e^{- n\widetilde{\varepsilon}_n^2 (c_2+4) }, \label{c2}\\
 \Pi\left( z^{(w)}\in\Theta: \mathrm{KL}(\pp_{\lambda_0},\pp_{z^{(w)}})\leq \widetilde{\varepsilon}_n^2 , \mathrm{V} ( \pp_{\lambda_0} , \pp_{z^{(w)}} ) \leq \widetilde{\varepsilon}_n^2 \right) & \geq c_4 e ^{-c_2 n\widetilde{\varepsilon}_n^2} \label{c3}.
\end{align}
Then, for $\varepsilon_n=\max(\overline{\varepsilon}_n,\widetilde{\varepsilon}_n)$ and a large enough constant $M>0,$ we have that
\begin{equation}
\label{postrate}
\Pi( z^{(w)}\in\Theta : \rho(\lambda_0,z^{(w)})\geq M\varepsilon_n| X^{(n)} ) \rightarrow 0
\end{equation}
in $\mathbb{P}_{\lambda_0}^{(n)}$-probability.
\end{thm}

\begin{rem}
\label{remrate}
Note that the posterior contraction rate $\varepsilon_n$ from Theorem \ref{thm2.1ghosal01} is not uniquely defined. If $\varepsilon_n$ is a posterior contraction rate, then so is, for instance, $(3-\sin(n))\varepsilon_n$ as well, or in fact any sequence that converges to zero at a slower rate than $\varepsilon_n.$ In general we are interested in finding the `fastest' posterior contraction rate $\varepsilon_n,$ in the sense that \eqref{postrate} holds for this $\varepsilon_n$ and there is no other sequence $\varepsilon_n^{\prime}\rightarrow 0,$ such that $\lim_{n\rightarrow\infty} \varepsilon_n^{\prime} / \varepsilon_n=0,$ for which \eqref{postrate} still holds with $\varepsilon_n$ replaced by $\varepsilon_n^{\prime}$ (and perhaps the constant $M$ replaced by another constant $M^{\prime}$). \qed
\end{rem}

The conditions of Theorem \ref{thm2.1ghosal01} merit some discussion. We restrict ourselves to heuristic reasoning only: an in-depth discussion can be found in \cite{ghosal00}. The important conditions of the theorem are \eqref{c1} and \eqref{c3}. Since the covering number can be thought of as measuring the size of the model, condition \eqref{c1} says that in order to have posterior contraction rate $\varepsilon_n,$ the model should not be too big. Furthermore, condition \eqref{c3} tells us that in order to have the posterior contraction rate $\varepsilon_n,$ the prior $\Pi$ should put some minimal mass in the Kullback-Leibler type neighbourhoods of $\lambda_0.$ Finally, condition \eqref{c2} adds some additional flexibility: for our purposes it is enough to be understood in the sense that the sets $\Theta_n$ are almost the support of the prior. This condition often allows one to avoid too stringent assumptions on the parameter set $\Theta,$ such as, for instance, its compactness.

Next we need to find effective means for checking the fact that our model and the prior satisfy the conditions of Theorem \ref{thm2.1ghosal01}. To that end we will employ Theorem 2.1 from \cite{vdv08a}. The following concept is needed in its statement: for a function ${w}_0:\mathcal{X}\rightarrow\mathbb{R}$ define the function $\phi_{{w}_0}:\mathbb{R}_{+}\rightarrow\mathbb{R}$ through
\begin{equation}
\label{cf}
\phi_{{w}_0}(\varepsilon)=\inf_{h\in\mathbb{H}:\| h - {w}_0 \|_{\infty} <\varepsilon } \frac{1}{2}\|h\|_{\mathbb{H}}^2 - \log \widetilde{\mathbb{P}}(\|W\|_{\infty}<\varepsilon).
\end{equation}
This is called the concentration function of the Gaussian process $W.$

\begin{thm}[\cite{vdv08a}]
\label{thm2.1vdv08a}
Let ${w}_0$ be contained in the support of $W.$ For any sequence of positive numbers $\hat{\varepsilon}_n>0$ satisfying
\begin{equation}
\label{phi}
\phi_{{w}_0}(\hat{\varepsilon}_n)\leq n\hat{\varepsilon}_n^2
\end{equation}
and any constant $C>1$ with $\exp(-Cn\hat{\varepsilon}_n^2)<1/2,$ there exist measurable sets $B_n\subset\mathbb{B},$ such that
\begin{align}
\log N(3\hat{\varepsilon}_n,B_n,\|\cdot\|_{\infty}) & \leq 6 C n\hat{\varepsilon}_n^2,\label{w1}\\
\widetilde{\mathbb{P}}(W\notin B_n) & \leq e^{-C n\hat{\varepsilon}_n^2},\label{w2}\\
\widetilde{\mathbb{P}} ( \| W - {w}_0 \|_{\infty} < 2 \hat{\varepsilon}_n ) & \geq e ^{-n\hat{\varepsilon}_n^2} \label{w3}.
\end{align}
\end{thm}

Comparing the three conditions \eqref{c1}--\eqref{c3} from Theorem \ref{thm2.1ghosal01} to the three conditions \eqref{w1}--\eqref{w3} from Theorem \ref{thm2.1vdv08a}, we see that they are of a similar type. Once we bridge the Hellinger distance, the Kullback-Leibler divergence and the divergence $\mathrm{V}$ appearing in Theorem \ref{thm2.1ghosal01} with the $\|\cdot\|_{\infty}$-distance, Theorems \ref{thm2.1ghosal01} and \ref{thm2.1vdv08a} will yield the posterior contraction rate.

The following lemma serves as the key for the purpose of bounding the divergences appearing in Theorem \ref{thm2.1ghosal01}. Its proof is found in Appendix \ref{appendix}.

\begin{lemma}
\label{lemma1}
Let $\lambda_1(x)=g(w(x))$ and $\lambda_2(x)=g(v(x))$ for $w,v\in\ell^{\infty}(\mathcal{X})$ and a Lipschitz continuous function $g:\rr\rightarrow[\kappa,\infty)$ with Lipschitz constant $\overline{g}.$ Then
\begin{enumerate}[(i)]
\item $ h(\mathbb{P}_{\lambda_1},\mathbb{P}_{\lambda_2}) \leq \frac{\overline{g}}{\sqrt{\kappa}} \| w -v \|_{\infty} ;$
\item $ \mathrm{KL}(\mathbb{P}_{\lambda_1},\mathbb{P}_{\lambda_2}) \leq \frac{\overline{g}^2}{\kappa} \| w -v \|_{\infty}^2 ;$
\item $ \mathrm{V}(\mathbb{P}_{\lambda_1},\mathbb{P}_{\lambda_2}) \leq \frac{\overline{g}^2}{\kappa} \|w-v\|_{\infty}^2\left(1+\frac{\overline{g}}{\kappa} \|w-v\|_{\infty} \right). $
\end{enumerate}
\end{lemma}

The following is our main result.

\begin{thm}
\label{mainthm}
Let $g:\rr\rightarrow[\kappa,\infty)$ be a Lipschitz continuous function with a Lipschitz constant $\overline{g},$ such that $g$ is invertible. Let the true intensity function $\lambda_0\geq \kappa$ be such that $w_0=g^{-1}(\lambda_0(\cdot))$ is contained in the support of the Gaussian process $W$ with continuous sample paths. Suppose the prior $\Pi$ is the law of the process $Z^{(W)}=(Z^{(W)}(x):x\in\mathcal{X})$ for $Z^{(W)}(x)=g(W(x)).$  Then for a sequence $\varepsilon_n=\hat{\varepsilon}_n\rightarrow 0$ satisfying the assumptions of Theorem \ref{thm2.1vdv08a} and a sufficiently large constant $M>0,$ the posterior distribution for $\lambda_0$ relative to the prior $\Pi$ satisfies
\begin{equation*}
\Pi(z^{(w)}\in\Theta:\rho(\lambda_0,z^{(w)})>M\varepsilon_n|X^{(n)})\rightarrow 0
\end{equation*}
in $\mathbb{P}_{\lambda_0}^{(n)}$-probability.
\end{thm}

\begin{proof}
For $B_n$ as in Theorem \ref{thm2.1vdv08a}, set $\Theta_n=\{ z^{(w)}:w\in B_n \}.$ We need to verify the conditions of Theorem \ref{thm2.1ghosal01}. Denote
\begin{equation*}
c_g=\frac{\overline{g}^2}{\kappa}+\frac{\overline{g}^3}{\kappa^2},
\end{equation*}
and let the constant $C>1$ from Theorem \ref{thm2.1vdv08a} be large enough, so that
\begin{equation*}
\frac{1}{ 4 c_g } \leq\frac{ C }{ 4 c_g } - 4.
\end{equation*}
Let $\hat{\varepsilon}_n$ be a sequence of positive numbers satisfying the conditions of Theorem \ref{thm2.1vdv08a}. Take $\overline{\varepsilon}_n=3\kappa^{-1/2}\overline{g}\hat{\varepsilon}_n.$ By Lemma \ref{lemma1} (i) and by inequality \eqref{w1},
\begin{equation*}
\log N(\overline{\varepsilon}_n,\Theta_n,\rho) \leq \log N(3\hat{\varepsilon}_n,B_n,\|\cdot\|_{\infty}) \leq 6 C n\hat{\varepsilon}_n^2=\frac{2\kappa C}{3\overline{g}^2} n\overline{\varepsilon}_n^2,
\end{equation*}
which verifies \eqref{c1} for the constant $c_1=2\kappa C/(3\overline{g}^2).$ Furthermore,
for $n$ large enough, so that $\widetilde{\varepsilon}_n$ is small, Lemma \ref{lemma1} (ii)--(iii) yields that
\begin{equation*}
\left\{ z^{(w)} \in\Theta: \mathrm{KL}(\pp_{\lambda_0},\pp_{z^{(w)}})\leq \widetilde{\varepsilon}_n^2 , \mathrm{V} ( \pp_{\lambda_0},\pp_{z^{(w)}} ) \leq \widetilde{\varepsilon}_n^2  \right\} \supset \{ w: c_g \| w - w_0 \|_{\infty}^2 \leq \widetilde{\varepsilon}_n^2 \}.
\end{equation*}
Set $\widetilde{\varepsilon}_n= 2 \sqrt{c_g} \hat{\varepsilon}_n.$ It follows from the above display and \eqref{w3} that
\begin{align*}
\Pi\left( z^{(w)}\in\Theta: \mathrm{KL}(\pp_{\lambda_0},\pp_{z^{(w)}})\leq \widetilde{\varepsilon}_n^2 , \mathrm{V} ( \pp_{\lambda_0},\pp_{z^{(w)}} ) \leq \widetilde{\varepsilon}_n^2 \right) & \geq \widetilde{\mathbb{P}} ( \| W - w_0 \|_{\infty} < 2 \hat{\varepsilon}_n ) \\
& \geq e ^{-n\hat{\varepsilon}_n^2}\\
& = \exp \left( - \frac{1}{4 c_g}n\widetilde{\varepsilon}_n^2\right).
\end{align*}
This verifies \eqref{c3} for $c_4=1$ and $c_2\geq 1/(4c_g).$ Finally, by \eqref{w2},
\begin{equation*}
\Pi( \Theta\setminus \Theta_n ) = \widetilde{\mathbb{P}}(W\notin B_n) \leq e^{-C n\hat{\varepsilon}_n^2}=\exp\left( -\frac{C}{4c_g}n\widetilde{\varepsilon}_n^2 \right).
\end{equation*}
This verifies \eqref{c2} for $c_3=1$ and $c_2 \leq C/(4 c_g)-4.$ Theorem \ref{thm2.1ghosal01} then yields the posterior contraction rate $\varepsilon_n= \max(\overline{\varepsilon}_n,\widetilde{\varepsilon}_n).$ Since both $\overline{\varepsilon}_n$ and $\widetilde{\varepsilon}_n$ are proportional to $\hat{\varepsilon}_n,$ we can simply take $\varepsilon_n=\hat{\varepsilon}_n$ and absorb the constants in the constant $M$ in the statement of Theorem \ref{thm2.1ghosal01}. This completes the proof.
\end{proof}

\begin{rem}
\label{rem1} An assumption made in Theorem \ref{mainthm} that the constant $\kappa$ is also a lower bound for the true intensity function $\lambda_0$ appears artifical, for it seems to require some knowledge of the unknown function $\lambda_0.$ However, there is a simple fix to that: all one has to do is to add independent realisations $Y_1,\ldots,Y_n$ of a Poisson point process $Y$ with intensity $\kappa\mu$ to original observations $X_1,\ldots,X_n.$ The new observations $X_i^{\prime}=X_i+Y_i$ will be again realisations of a Poisson point process, but now with the intensity function $\kappa+\lambda_0.$ Inferential conclusions drawn on $\kappa+\lambda_0$ in this new model can then be directly transformed into conclusions on $\lambda_0.$ \qed
\end{rem}

\begin{rem}
\label{rem1a} The assumptions on the function $g$ from Theorem \ref{mainthm} are sufficient, but not necessary. For instance, the conclusions of Theorem \ref{mainthm} also hold true for the function $g(x)=\kappa+|x|$ that is Lipschitz (with constant $\overline{g}=1$), but not invertible. This requires certain modification of the arguments in the proof. Assume that $\lambda_0=\kappa+\overline{\lambda}_0$ for a function $\overline{\lambda}_0\geq 0$ contained in the support of $W$. Then for $w_0$ in Theorem \ref{thm2.1vdv08a} taken to be $\overline{\lambda}_0$, equations \eqref{w1}--\eqref{w3} hold for some sieves $B_n$. Define sieves $\Theta_n$ as in the proof of Theorem \ref{mainthm}: $\Theta_n=\{ z^{(w)}:w\in B_n \}.$ Then condition \eqref{c1} can be verified as in the proof of Theorem \ref{mainthm}. To show \eqref{c2}, it suffices to observe that $\Pi( \Theta\setminus \Theta_n ) \leq \widetilde{\mathbb{P}}(W\notin B_n)$. Finally, \eqref{c3} follows as in the proof of Theorem \ref{mainthm}. \qed
\end{rem}

\begin{rem}
\label{rem2}
Motivated by applications of the so-called log-Gaussian Cox processes, one could have argued that a reasonable prior $\Pi$ for $\lambda_0$ would have been the process $Z^{(W)}=(Z^{(W)}(x):x\in\mathcal{X})$ defined through
\begin{equation*}
Z^{(W)}(x)=g(W(x))=e^{W(x)}, \quad x\in\mathcal{X}.
\end{equation*}
This transforms $W$ into a process $Z^{(W)}$ with strictly positive sample paths. However, the function $g(x)=e^x$ does not satisfy assumptions of our theorem. Examination of the proof of Lemma \ref{lemma1} shows that in this case there does not seem to exist a good way to control the probability divergences in the statement of Theorem \ref{thm2.1ghosal01} in terms of the $\|\cdot\|_{\infty}$-distance. This then does not permit to invoke Theorem \ref{thm2.1vdv08a} in order to derive the posterior contraction rate. The possibility that for such a prior the posterior contracts at a suboptimal rate (in a sense that there exists some other prior, for which the posterior contracts at a faster rate $\epsilon_n^{\prime};$ cf.\ Remark \ref{remrate}) is not to be discarded.  \qed
\end{rem}

\begin{rem}
\label{rem3a}
Practical implementation of the non-parametric Bayesian approach lies outside the scope of the present work. We only remark that provided the function $g$ is not only Lipschitz, but also bounded, our approach can be implemented along the lines in \cite{adams09}; see especially Section 5.3 there. A typical choice of the function $g$ could be e.g.\ a minor variation on the logistic function,
\begin{equation*}
g(x)=\kappa+ \frac{g^{\ast}}{1+e^{-x}},
\end{equation*}
where $g^{\ast}>0$ is a constant. In a practical implementation one could take $\kappa=0.$
\qed
\end{rem}

\begin{rem}
\label{rem3}
An interesting statistical problem related to the one we are considering in this note is non-parametric estimation of the intensity function of a cyclic Poisson point processes   (i.e.\ Poisson point processes with periodic intensity functions) over $\mathcal{X}=[0,T]$. A recent reference dealing with estimation of the unknown period in this model is \cite{belitser13}. \qed
\end{rem}

\begin{rem}
\label{remharry}
After the completion of this work, we learnt about a related paper \cite{vz13}, that deals with non-parametric Bayesian estimation of the periodic intensity function of an inhomogeneous Poisson process on the real line. The main theorems in that paper are of the same nature as ours, but here we want to underline some essential differences. Firstly, our approach also covers the estimation of spatial point processes. In our analysis we assume that i.i.d. copies of a Poisson point process are at our disposal, whereas in \cite{vz13} this assumption has not been made. However, the assumption of periodicity of the intensity function is almost the same as having i.i.d.\ observations on a time interval $[0,T]$, where $T$ is the period of the intensity function. So, essentially our results are applicable to a wider class of point processes. On the other hand, \cite{vz13} also analyse the case of discrete observations and for this case they give results on the estimation of integrated intensity functions, instead of the intensity functions themselves. Studying integrated intensity functions originates from the fact that the intensities themselves are not identifiable under a discrete time observations regime. Another difference between their paper and ours is in the choice of the prior: \cite{vz13} work with a free-knot spline prior, while ours is based on a fixed transformation of a Gaussian process. More subtle differences can be observed by comparing the two papers in detail. See also the PhD thesis \cite{serra13}, that contains some further related results. \qed
\end{rem}

\begin{rem}
\label{remalice}
Another related paper of whose existence we became aware of is \cite{kirichenko15}. It deals with derivation of the posterior contraction rate for the method in \cite{adams09}, and its results are closely related to ours. The most important difference between our work and \cite{kirichenko15} is that in Theorem \ref{mainthm} and Remark \ref{rem1a} we allow a non-smooth link function $g$, whereas a logistic function is used in \cite{kirichenko15}. \qed
\end{rem}

\section{Examples of the prior}
\label{priors}

In this section we consider two concrete examples of the prior and compute the posterior contraction rate for them explicitly. We will restrict our attention to the case $\mathcal{X}=[0,1]^d,$ which is a practically relevant one. The results in this section are consequences of Theorem \ref{mainthm} and probabilistic results on the behaviour of various Gaussian processes.

For a multi-index $\alpha=(\alpha_1,\ldots,\alpha_d)$ we let $|\alpha|=\sum_i \alpha_i$ and introduce the partial derivative operator
\begin{equation*}
D^{\alpha}=\frac{\partial^{|\alpha|}}{\partial x_1^{\alpha_1} \cdots \partial x_d^{\alpha_d} }.
\end{equation*}
We will call a continuous function $\lambda:\mathcal{X}\rightarrow\mathbb{R}$ $\beta$-H\"older regular for $\beta\geq 1,$ if $D^{\alpha} \lambda$ is continuous for any $|\alpha|\leq\beta.$ We will denote the space of $\beta$-H\"older-regular functions by $\mathcal{C}^{\beta}(\mathcal{X})$ (equipped with the uniform norm $\|\cdot\|_{\infty}$). Furthermore, $\mathcal{C}(\mathcal{X})$ will denote the space of continuous functions on $\mathcal{X}$ (equipped with the uniform norm $\|\cdot\|_{\infty}$).

First let $d=1.$ 

\begin{exam}
\label{example1}
Let $\overline{W}=(\overline{W}(x):x\in\mathcal{X})$ be a standard Brownian motion over the time interval $\mathcal{X}=[0,1]$ and let $\eta_0,\eta_1,\ldots,\eta_{\beta}$ be standard normal random variables. Assume that $\eta_0,\eta_1,\ldots,\eta_{\beta},\overline{W}$ are independent. The modified Riemann-Liouville process $W=(W(x):x\in\mathcal{X})$ with Hurst parameter $\beta>0$ is defined as
\begin{equation*}
W(x)=\sum_{k=0}^{\beta}\eta_k x^k + \int_0^x (x-y)^{\beta-1/2}\mathrm{d}\overline{W}_y, \quad x\in\mathcal{X},
\end{equation*}
see Section 4.2 in \cite{vdv08a}. Our prior $\Pi$ will be the law of the process $Z^{(W)}=(Z^{(W)}(x):x\in\mathcal{X})$ defined by \eqref{pz}. By Theorem 4.3 in \cite{vdv08a}, the support of $W$ is the whole space $\mathcal{C}(\mathcal{X}),$ and if $w_0\in\mathcal{C}^{\beta}(\mathcal{X}),$ then $\phi_{w_0}(\varepsilon)\asymp \varepsilon^{-1/\beta}$ as $\varepsilon\downarrow 0.$ It then follows from Theorem \ref{mainthm} by solving inequality \eqref{phi} (cf.\ \cite{vdv08a}, pp.\ 1449--1450) that the posterior contracts at the rate $n^{-\beta/(2\beta+1)}.$ This is the minimax estimation rate for a $\beta$-H\"older-regular function in a variety of non-parametric estimation problems. See in particular Theorem 6.5 in \cite{kutoyants98} for the Poisson point processes setting. The rate $n^{-\beta/(2\beta+1)}$ can thus be thought of as an optimal posterior contraction rate in this particular setting. \qed
\end{exam}

Now we consider the general case $d\geq 1$ and simultaneously illustrate the fact that the statement of Theorem \ref{mainthm} holds not only for Gaussian processes, but also for certain conditionally Gaussian processes.

\begin{exam}
\label{example2}
Let $W=(W(x):x\in\mathbb{R}^d)$ be a centred homogeneous Gaussian random field satisfying the properties listed in Theorem 3.1 in \cite{vdv09}. Consider a positive random variable $A$ defined on the same probability space as $W$ and independent of it, and assume that its distribution satisfies the requirement in formula (3.4) in \cite{vdv09}. In fact, for concreteness we take $A$ such that $A^d$ possesses a Gamma distribution (this is a popular choice in practical implementations of Bayesian procedures in various statistical problems). Now consider the restriction $W^A$ of the rescaled process $x\mapsto W(Ax)$ to $\mathcal{X}.$ Conditional on $A,$ the process $W^A$ is again a Gaussian process. It is shown in Theorem 3.1 in \cite{vdv09} that for all $n$ large enough, provided $w_0\in\mathcal{C}^{\beta}(\mathcal{X}),$ there exist Borel subsets $B_n$ of $\mathcal{C}(\mathcal{X}),$ such that the inequalities
\begin{align}
\log N({\varepsilon}_{n,1},B_n,\|\cdot\|_{\infty}) & \leq n{\varepsilon}_{n,1}^2,\label{w1b}\\
\widetilde{\mathbb{P}}(W^A\notin B_n) & \leq e^{- 4 n{\varepsilon}_{n,2}^2},\label{w2b}\\
\widetilde{\mathbb{P}} ( \| W^A - {w}_0 \|_{\infty} \leq {\varepsilon}_{n,2} ) & \geq e ^{-n{\varepsilon}_{n,2}^2},\label{w3b}
\end{align}
hold, where
\begin{align*}
\varepsilon_{n,1} &= K \varepsilon_{n,2} (\log n)^{\kappa_2}, \quad \kappa_2 =\frac{1+d}{2},\\
\varepsilon_{n,2} &= n^{-\beta/(2\beta+d)} (\log n)^{\kappa_1}, \quad \kappa_1=\frac{1+d}{2+d/\beta},
\end{align*}
and $K$ is a sufficiently large constant. Our prior $\Pi$ will be the law of the process $Z^{(W^A)}=( Z^{(W^A)}(x):x\in\mathcal{X} ),$ where $Z^{(W^A)}(x)=g( W^A(x) ).$ The inequalities \eqref{w1b}--\eqref{w3b} are analogues of those in Theorem \ref{thm2.1vdv08a}. Also an analogue of \eqref{cf} is satisfied thanks to formula (5.1) in
\cite{vdv09} and subsequent arguments there.

As the proof of Theorem \ref{mainthm} depends only on the inequalities \eqref{w1b}--\eqref{w3b}, a minor modification of the arguments shows that the posterior in this case contracts at the rate
\begin{equation}
\label{rt2}
\varepsilon_n=n^{-\beta/(2\beta+d)} (\log n)^{(1+d)(4\beta+d)/(4\beta+2d)}.
\end{equation}
In fact, the following choice of the quantities involved in the proof of Theorem \ref{mainthm} works:
\begin{gather*}
\overline{\varepsilon}_n=\frac{ \overline{g} }{ \sqrt{\kappa} } \varepsilon_{n,1}, \quad \widetilde{\varepsilon}_{n,2}=\sqrt{c_g}\varepsilon_{n,2},\\
c_1=\frac{\kappa}{\overline{g}^2}, \quad \frac{1}{c_g}\leq c_2 \leq \frac{4}{c_g}-4, \quad c_3=1, \quad c_4=1.
\end{gather*}

The posterior contraction rate in \eqref{rt2} differs from the minimax estimation rate in Theorem 6.5 in \cite{kutoyants98} only in a logarithmic factor, which is unimportant for all practical purposes. Whether the logarithmic loss is an artifact of a specific proof of the result in \cite{vdv09} studying the properties of the rescaled process $W^A,$ or is essential, is not entirely clear. More important, however, is the fact that if $g$ is such that $g^{-1}$ is infinitely many times differentiable with suitably bounded derivatives, our arguments show that even without knowing the true smoothness order $\beta$ of $\lambda_0,$ the Bayesian approach will automatically attain the optimal posterior contraction rate (up to a logarithmic factor): the prior $\Pi$ is constructed in a way not utilising information on $\beta.$ In other words, the Bayesian approach is adaptive  in this case. This is in contrast to the kernel method from \cite{kutoyants98}, which requires knowledge of $\beta$ for an optimal selection of the smoothing parameter. \qed
\end{exam}

The two examples of Gaussian process based priors we considered in this section obviously do not exhaust all possible examples where our results are applicable. For instance, we mention the fact after another slight modification of the statement of Theorem 3 one can also cover the case of Gaussian tensor-product spline priors studied in \cite{dejonge12}. Again, up to a logarithmic factor, the posterior will contract at the optimal rate. We omit the details.

\appendix
\section{  }
\label{appendix}

\begin{proof}[Proof of Lemma \ref{lemma1}]
Part (i) follows from part (ii) and the well-known inequality
\begin{equation*}
h^2(\mathbb{P}_{\lambda_1},\mathbb{P}_{\lambda_2}) \leq \mathrm{KL}(\mathbb{P}_{\lambda_1},\mathbb{P}_{\lambda_2})
\end{equation*}
between the squared Hellinger distance and the Kullback-Leibler divergence.

We prove part (ii). Using Proposition 6.11 in \cite{karr86} or Theorem 1.3 from \cite{kutoyants98}, as well as  Lemma 1.1 there, we have
\begin{equation}
\label{kl1}
\mathrm{KL}(\mathbb{P}_{\lambda_1},\mathbb{P}_{\lambda_2})=\int_{ \mathcal{X} } \lambda_1(x) \log \left( \frac{\lambda_1(x)}{\lambda_2(x)} \right) \mathrm{d}\mu(x) - \int_{ \mathcal{X} } \left\{\frac{\lambda_1(x)}{\lambda_2(x)}-1\right\}\lambda_2(x)\mathrm{d}\mu(x).
\end{equation}
Now since $\log(1+x)\leq x$ for $x>-1,$ we get that
\begin{align*}
\mathrm{KL}(\mathbb{P}_{\lambda_1},\mathbb{P}_{\lambda_2}) & \leq \int_{ \mathcal{X} } [ \lambda_1(x) - \lambda_2(x) ]^2\frac{1}{\lambda_2(x)}\mathrm{d}\mu(x)\\
& \leq \frac{1}{\kappa} \int_{ \mathcal{X} } [ \lambda_1(x) - \lambda_2(x) ]^2\mathrm{d}\mu(x)\\
& \leq \frac{1}{\kappa} \| \lambda_1 - \lambda_2 \|_{\infty}^2\\
& \leq \frac{\overline{g}^2}{\kappa} \| w - v \|_{\infty}^2.
\end{align*}
This proves part (ii). Here we also see the role of the constant $\kappa>0.$

We prove part (iii). Letting $U \sim \mathbb{P}_{\lambda_1}$ and denoting by $\ex_{{\lambda_1}}[\cdot]$ the expectation under $\mathbb{P}_{\lambda_1},$ we have
\begin{equation}
\label{v1}
\mathrm{V}(\mathbb{P}_{\lambda_1},\mathbb{P}_{\lambda_2}) = \ex_{{\lambda_1}}\left[ \log^2 \left( \frac{\mathrm{d}\mathbb{P}_{\lambda_1}}{\mathrm{d}\mathbb{P}_{\lambda_2}}(U) \right)  \right] - \mathrm{KL}^2(\mathbb{P}_{\lambda_1},\mathbb{P}_{\lambda_2}).
\end{equation}
Again, using Proposition 6.11 in \cite{karr86} or Theorem 1.3 from \cite{kutoyants98}, as well as  Lemma 1.1 there, combined with formula \eqref{kl1} above, after some long and uninspiring computations we get from \eqref{v1} that
\begin{align*}
\mathrm{V}(\mathbb{P}_{\lambda_1},\mathbb{P}_{\lambda_2}) & =  \int_{ \mathcal{X} } \lambda_1(x) \log^2 \left( \frac{\lambda_1(x)}{\lambda_2(x)} \right) \mathrm{d}\mu(x)  \\
&=\int_{ \lambda_1<\lambda_2 } \lambda_1(x) \log^2 \left( \frac{\lambda_1(x)}{\lambda_2(x)} \right) \mathrm{d}\mu(x)\\
& + \int_{ \lambda_1>\lambda_2 } \lambda_1(x) \log^2 \left( \frac{\lambda_1(x)}{\lambda_2(x)} \right) \mathrm{d}\mu(x)\\
&=\mathrm{I}_1+\mathrm{I}_2,
\end{align*}
with an obvious definition of $\mathrm{I}_1$ and $\mathrm{I}_2.$ Recall the elementary inequality
\begin{equation*}
\frac{x}{1+x}\leq \log(1+x)\leq x, \quad x>-1.
\end{equation*}
This inequality gives that on the set $\{\lambda_1<\lambda_2\},$
\begin{equation*}
\log^2 \left( \frac{\lambda_1(x)}{\lambda_2(x)} \right) \leq \frac{1}{\lambda_1^2(x)}[ \lambda_1(x)-\lambda_2(x) ]^2.
\end{equation*}
Hence
\begin{equation*}
\mathrm{I}_1 \leq \frac{1}{\kappa} \| \lambda_1-\lambda_2 \|_{\infty}^2\leq\frac{\overline{g}^2}{\kappa} \| w-v \|_{\infty}^2.
\end{equation*}
On the other hand, on the set $\{\lambda_1>\lambda_2\},$
\begin{equation*}
\log^2 \left( \frac{\lambda_1(x)}{\lambda_2(x)} \right) \leq \left( \frac{\lambda_1(x)}{\lambda_2(x)} - 1\right)^2.
\end{equation*}
Therefore,
\begin{align*}
\mathrm{I}_2 & \leq \int_{ \lambda_1>\lambda_2 } [ \lambda_1(x) - \lambda_2(x)  ]^2 \frac{\lambda_1(x)}{\lambda_2^2(x)} \mathrm{d}\mu(x)\\
& = \int_{ \lambda_1>\lambda_2 } [ \lambda_1(x) - \lambda_2(x)  ]^3 \frac{1}{\lambda_2^2(x)} \mathrm{d}\mu(x)\\
&+\int_{ \lambda_1>\lambda_2 } [ \lambda_1(x) - \lambda_2(x)  ]^2\frac{1}{\lambda_2(x)} \mathrm{d}\mu(x)\\
&\leq \frac{1}{\kappa^2} \|\lambda_1-\lambda_2\|^3_{\infty}+\frac{1}{\kappa} \|\lambda_1-\lambda_2\|^2_{\infty}\\
&\leq \frac{\overline{g}^2}{\kappa} \|w-v\|_{\infty}^2\left(1+\frac{\overline{g}}{\kappa} \|w-v\|_{\infty} \right).
\end{align*}
This completes the proof of part (iii) and hence of the lemma too.
\end{proof}

\bibliographystyle{plainnat}

\begin{thebibliography}{99}
\bibitem[Adams et al.(2009)]{adams09} R.P.\ Adams, I.\ Murray and D.J.C.\ MacKay.\ Tractable nonparametric Bayesian inference in Poisson processes with Gaussian process intensities. \emph{Proceedings of the 26th Annual International Conference on Machine Learning}, 9--16. ACM, New York, NY, 2009.
\bibitem[Belitser et al.(2012)]{belitser13} E.\ Belitser, P.\ Serra and H.\ van Zanten.\ Estimating the period of a cyclic non-homogeneous Poisson process. \emph{Scand.\ J.\ Stat.}, 40:204--218, 2013.
\bibitem[Belitser et al.(2013)]{vz13} E.\ Belitser, P.\ Serra and H.\ van Zanten.\ Rate-optimal Bayesian intensity smoothing for inhomogeneous Poisson processes.\ arXiv:1304.6017 [math.ST], 2013.
\bibitem[Diaconis and Freedman(1986)]{diaconis86} P.\ Diaconis and D.\ Freedman. On the consistency of Bayes estimates. With a discussion and a rejoinder by the authors. \emph{Ann. Statist.}, 14:1--67, 1986.
\bibitem[Diggle(1985)]{diggle85} P.\ Diggle. A kernel method for smoothing point process data. \emph{App. Statist.}, 34:138--147, 1985.
\bibitem[Ghosal et al.(2000)]{ghosal00} S.\ Ghosal, J.K.\ Ghosh and A.W.\ van der Vaart. Convergence rates of posterior distributions. \emph{Ann.\ Statist.}, 28:500--531, 2000.
\bibitem[Ghosal and van der Vaart(2001)]{ghosal01} S.\ Ghosal and A.W.\ van der Vaart. Entropies and rates of convergence for maximum likelihood and Bayes estimation for mixtures of normal densities.
\emph{Ann.\ Statist.}, 29:1233--1263, 2001. 
\bibitem[Ghosal and van der Vaart(2007)]{ghosal07} S.\ Ghosal and A.W.\ van der Vaart. Convergence rates of posterior distributions for non-i.i.d.\ observations. \emph{Ann.\ Statist.}, 35:192--223, 2007.
\bibitem[Gugushvili et al.(2018)]{gugu18} S.~Gugushvili, F. van der Meulen, M. Schauer and P. Spreij. Fast and scalable non-parametric Bayesian inference for Poisson point processes. arXiv:1804.03616 [stat.ME], 2018.
\bibitem[Heikkinen and Arjas(1998)]{arjas98} J.\ Heikkinen and E.\ Arjas. Non-parametric Bayesian estimation of a spatial Poisson intensity. \emph{Scand.\ J.\ Statist.}, 25:435--450.
\bibitem[de Jonge and van Zanten(2012)]{dejonge12} R.\ de Jonge and J.H.\ van Zanten.\ Adaptive estimation of multivariate functions using conditionally Gaussian tensor-product spline priors. \emph{Electron.\ J.\ Stat.}, 6:1984--2001, 2012.
\bibitem[Karr(1986)]{karr86} A.F.\ Karr.\ \emph{Point Processes and their Statistical Inference.} Probability: Pure and Applied, 2. Marcel Dekker, Inc., New York, 1986.
\bibitem[Kingman(1993)]{kingman93} J.F.C.\ Kingman. \emph{Poisson Processes.} Oxford Studies in Probability, 3. Oxford Science Publications. The Clarendon Press, Oxford University Press, New York, 1993.
\bibitem[Kirichenko and van Zanten(2015)]{kirichenko15} A.~Kirichenko and J.H.~van Zanten. Optimality of Poisson processes intensity learning with Gaussian processes. \emph{J. Mach. Learn. Res.}, 16:2909--2919, 2015.
\bibitem[Kutoyants(1998)]{kutoyants98} Yu.A.\ Kutoyants. \emph{Statistical Inference for Spatial Poisson Processes.} Lecture Notes in Statistics, 134. Springer-Verlag, New York, 1998.
\bibitem[Lifshits(2012)]{lifshits12} M. Lifshits.\ \emph{Lectures on Gaussian Processes.} Springer Briefs in Mathematics. Springer, Heidelberg, 2012.
\bibitem[Lo(1982)]{lo82} A.Y.\ Lo.\ Bayesian nonparametric statistical inference for Poisson point processes.\ 
\emph{Z.\ Wahrsch.\ Verw.\ Gebiete}, 59:55--66, 1982.
\bibitem[M{\o}ller et al.(1998)]{moller98} J.\ M{\o}ller, A.R.\ Syversveen and R.P.\ Waagepetersen. Log Gaussian Cox processes. \emph{Scand.\ J.\ Statist.}, 25:451--482, 1998.
\bibitem[M{\o}ller and Waagepetersen(2007)]{moller07} J.\ M{\o}ller and R.P.\ Waagepetersen.\ Modern statistics for spatial point processes. \emph{Scand.\ J.\ Statist.}, 34:643--684, 2007.
\bibitem[Rasmussen and Williams(2006)]{rasmussen06} C.E.\ Rasmussen and C.K.I.\ Williams.\ \emph{Gaussian Processes for Machine Learning.} Adaptive Computation and Machine Learning. MIT Press, Cambridge, MA, 2006.
\bibitem[Serra(2013)]{serra13} P.J.\ de Andrade Serra. \emph{Non-parametric Inference and Tracking for Poisson Processes}. PhD thesis, Eindhoven, Technische Universiteit Eindhoven, 2013. Available online at \url{http://dx.doi.org/10.6100/IR758267}.
\bibitem[van der Vaart and van Zanten(2008a)]{vdv08a}  A.W.\ van der Vaart and J.H.\ van Zanten. Rates of contraction of posterior distributions based on Gaussian process priors. \emph{Ann.\ Statist.}, 36:1435--1463, 2008a.
\bibitem[van der Vaart and van Zanten(2008b)]{vdv08b} A.W.\ van der Vaart and J.H.\ van Zanten. Reproducing kernel Hilbert spaces of Gaussian priors. \emph{Pushing the Limits of Contemporary Statistics: Contributions in Honor of Jayanta K. Ghosh}, 200--222, Inst. Math. Stat. Collect., 3. Inst. Math. Statist., Beachwood, OH, 2008b.
\bibitem[van der Vaart and van Zanten(2009)]{vdv09} A.W.\ van der Vaart and J.H.\ van Zanten.\ Adaptive Bayesian estimation using a Gaussian random field with inverse Gamma bandwidth. \emph{Ann.\ Statist.}, 37:2655--2675, 2009.
\end{thebibliography}

\end{document}